\documentclass[10pt,draft]{article}
\usepackage{amssymb,amsmath,amsthm,latexsym}

\theoremstyle{plain}
\newtheorem{theorem}{Theorem}[section]

\newtheorem{lemma}{Lemma}[section]
\newtheorem{corollary}{Corollary}[section]

\theoremstyle{remark}

\newtheorem{example}{Example}[section]

\newcommand{\R}{{\mathbb{R}}}

\newcommand{\Aff}{\mathrm{aff}\,}
\newcommand{\Bd}{\mathrm{bd}\,}
\newcommand{\Cl}{\mathrm{cl}\,}

\newcommand{\Dim}{\mathrm{dim\,}}
\newcommand{\Int}{\mathrm{int\,}}
\newcommand{\Lin}{\mathrm{lin\,}}
\newcommand{\Pos}{\mathrm{pos\,}}
\newcommand{\Rbd}{\mathrm{rbd\,}}
\newcommand{\Rint}{\mathrm{rint\,}}
\newcommand{\Span}{\mathrm{span\,}}

\newcommand{\fE}{\text{\textup{\textit{\tiny{E}}}}}
\newcommand{\fF}{\text{\textup{\textit{\tiny{F}}}}}
\newcommand{\fL}{\text{\textup{\textit{\tiny{L}}}}}
\newcommand{\fV}{\text{\textup{\textit{\tiny{V}}}}}
\newcommand{\fW}{\text{\textup{\textit{\tiny{W}}}}}

\begin{document}

\begin{center} {\Large \textbf{Moreau-Type Characterizations of Polar Cones}}

\bigskip

\textbf{Valeriu~Soltan}

Department of Mathematical Sciences, George Mason University

4400 University Drive, Fairfax, VA 22030, USA

vsoltan@gmu.edu

\end{center}

\begin{abstract}
A theorem of Moreau (1962) states that given a closed convex cone $C$ and its (negative) polar cone $C^\circ$ in a real Hilbert space $H$, vectors $y \in C$ and $z \in C^\circ$ are metric projections of a vector $u \in H$ on $C$ and $C^\circ$, respectively, if and only if they satisfy the following conditions: $y$ and $z$ are orthogonal and $u = y + z$. We show that these  conditions provide characteristic properties of polar cones $C$ and $C^\circ$ in the family of pairs of convex subsets of $H$ or $\R^n$. A related result on separation of $C$ a face of $C^\circ$ in $\R^n$ is proved.

\medskip


\medskip

\end{abstract}

\section{Introduction and Main Results} \label{1-intro}

A well known assertion of linear analysis states that given a closed subspace $S$ of a real Hilbert space $H$, every vector $u \in H$ is uniquely expressible as the sum $u = y + z$, where $y$ and $z$ are, respectively, the orthogonal projections of $u$ on $S$ and its orthogonal complement $S^\bot$. Moreau~\cite{mor62} obtained the following far-reaching generalization of this assertion in terms of metric projections on a closed convex cone $C \subseteq H$ and its polar cone $C^\circ$ (see Alber~\cite{alb00} for a further generalization of Moreau's theorem to the case of cones in Banach spaces).

\begin{theorem} \label{moreau}  \textnormal{(Moreau)}
Let\, $C \subseteq H$ be a closed convex cone, and let\, $u \in H$. For vectors\, $y \in C$ and\, $z \in C^\circ$, the following conditions are equivalent.
\begin{enumerate}

\item[$(a)$] $y$ and $z$ are, respectively, the metric projections of\, $u$ on $C$ and\, $C^\circ$.

\item[$(b)$] $y$ and $z$ are orthogonal and satisfy the equality $u = y + z$.

\end{enumerate}
\end{theorem}

One may ask whether Theorem~\ref{moreau} can be further generalized by replacing cones $C$ and $C^\circ$ with more general convex sets. Our main results (see Theorems~\ref{decomposition}--\ref{orthogonal} below) show that such an attempt leads to characteristic properties of convex cones which are polar of each other. In this regard, we provide a related result on separation of a cone $C \subseteq \R^n$ and a face of $C^\circ$ in $\R^n$ (see Theorem~\ref{aux}).

We recall that a nonempty convex set $C \subseteq H$ is a \textit{convex cone} if $\lambda x \in C$ whenever $\lambda \geqslant 0$ and $x \in C$. This definition implies that the origin $o\,$ of $H$ belongs to $C$ (indeed, $o = 0 x \in C$ for any $x \in C$), although a stronger condition $\lambda > 0$ can be beneficial; see, e.\,g., \cite{l-s16}.

The (negative) \textit{polar set} of a nonempty set $X \subseteq H$, denoted $X^\circ$, is defined by
\[
X^\circ = \{ x \in H : \langle x, e \rangle \leqslant  0 \ \ \forall\, e \in X \},
\]
where $\langle x, e \rangle$ stands for the scalar product of vectors $x$ and $e$. It is well known (and easy to prove) that the set $X^\circ$ is a closed convex cone for any choice of $X \subseteq H$.

In what follows, we will say that convex cones $C$ and $D$ in $H$ are \textit{polar of each other} if $C = D^\circ$ and $D = C^\circ$.

A classical result of linear analysis asserts that given a nonempty closed convex set $E \subseteq H$ and a vector $u \in H$, the set $E$ contains a unique vector $z$ which is nearest to $u$:
\[
\| u - z \| = \inf \{ \| u - x \| : x \in E \}.
\]
This nearest vector $z$ is denoted $p_\fE (u)$ and called the \textit{metric projection} of $u$ on $E$. Furthermore, if $u \in H \setminus E$, then $E$ lies in the closed halfspace
\[
\{ x \in H : \langle x - p_\fE (u), u - p_\fE (u) \rangle \leqslant  0 \}
\]
(see, for instance, Deutsch~\cite{deu01} for various results and references on this topic).

\begin{theorem} \label{decomposition}
Let $E$ and $F$ be nonempty closed convex sets in a real Hilbert space $H$. The following asser\-tions are equivalent.
\begin{enumerate}

\item[$(a)$] Every vector $u \in H$ is expressible as\, $u = p_\fE(u) + p_\fF(u)$.

\item[$(b)$] $E$ and $F$ are polar cones of each other.

\end{enumerate}
\end{theorem}

The remaining theorems are formulated for the case of the $n$-dimensional space $\R^n$; it is \textit{an open question} whether they are extendable to infinite dimension.

\begin{theorem} \label{projections}
Let $E$ and $F$ be nonempty closed convex sets in $\R^n$. The following assertions are equivalent.
\begin{enumerate}

\item[$(a)$] $E + F = \R^n$ and the metric projections $p_\fE (u)$ and $p_\fF (u)$ are orthogonal for every choice of\, $u \in \R^n$.

\item[$(b)$] $E$ and $F$ are polar cones of each other.

\end{enumerate}
\end{theorem}

Given an ordered pair of convex sets $E$ and $F$ and a vector $u \in \R^n$, expressions $u = y_1 + z_1$ and $u = y_2 + z_2$, where $y_1, y_2 \in E$ and $z_1, z_2 \in F$, will be called different if and only if $y_1 \ne y_2$ (equivalently, $z_1 \ne z_2$).

\begin{theorem} \label{orthogonal}
Let $E$ and $F$ be nonempty convex sets in $\R^n$. The following assertions are equivalent.
\begin{enumerate}

\item[$(a)$] Every vector $u \in \R^n$ is uniquely expressible as $u = y + z$, where $y$ and $z$ are orthogonal vectors from $E$ and $F$, respectively.

\item[$(b)$] $E$ and $F$ are polar cones of each other.

\end{enumerate}
\end{theorem}

One more result of a similar spirit characterizes complementary planes in $\R^n$. We recall that a \textit{plane} $L$ (a flat, or an affine subspace, in other terminology) of dimension $m$ in $\R^n$, $0 \leqslant  m \leqslant  n$, is a translate of an $m$-dimensional subspace $S$ of $\R^n$: $L = c + S$ for a suitable vector $c \in \R^n$. Planes $L_1$ and $L_2$ in $\R^n$ are called \textit{complementary} if they are translates of complementary subspaces $S_1$ and $S_2$:
\[
L_1 = c_1 + S_1, \ \ L_2 = c_2 + S_2, \ \ S_1 + S_2 = \R^n, \ \ S_1 \cap S_2 = \{ o \}.
\]

\begin{theorem} \label{slanted}
Let $E$ and $F$ be nonempty convex sets in $\R^n$. The following assertions are equivalent.
\begin{enumerate}

\item[$(a)$] Every vector $u \in \R^n$ is uniquely expressible as $u = y + z$, where $y \in E$ and $z \in F$.

\item[$(b)$] $E$ and $F$ are complementary planes.

\end{enumerate}
\end{theorem}

The remarks below illustrate Theorems~\ref{decomposition}--\ref{slanted}.
\begin{enumerate}

\item[1.] The assumption on closedness of both sets $E$ and $F$ is essential in Theorems~\ref{decomposition} and~\ref{projections} since it guarantees the existence of metric projections $p_\fE (u)$ and $p_\fF (u)$. On the other hand, the open unit ball $U = \{ x \in \R^n : \| x \| < 1 \}$ contains no vector nearest to a given vector $u \in \R^n \setminus U$.

\item[2.] Assertion $(a)$ in Theorem~\ref{projections} cannot be replaces with the following weaker version:

$(a')$ $E + F = \R^n$ and for every choice of\, $u \in \R^n$, the metric projections of $u$ on $\Cl E$ and $\Cl F$ are orthogonal. Indeed, let
\[
E = \{ (x, y) : x, y > 0 \} \cup \{ o \} \quad \text{and} \quad F = \{ (x, y) : x, y < 0 \} \cup \{ o \}.
\]
Clearly, both sets $E$ and $F$ are nonclosed convex cones with apex $o$. These cones satisfy assertion $(a')$ but are not polar of each other.

\item[3.] The assumption $E + F = \R^n$ is essential in assertion $(a)$ of Theorem~\ref{projections}. For instance, let $E$ and $F$ be closed intervals which lie, respectively, in the coordinate axes of the plane $\R^2$. Then $p_\fE (u) \bot p_\fF (u)$ for any vector $u \in \R^2$, while neither $E$ nor $F$ is a convex cone.

\item[4.] The assumption on uniqueness of expressions $u = y + z$ is essential in assertion $(a)$ of Theorem~\ref{orthogonal}. For instance, let $E$ and $F$ denote, respectively, the upper and lower closed halfplanes of $\R^2$ determined by the $x$-axis. Then any vector $u \in \R^2$ can be written as a sum of orthogonal vectors $y \in E$ and $z \in F$, while $E$ and $F$ are not polar cones of each other.

\item[5.] The previous remark shows that uniqueness of representations $u = y + z$ also is essential in assertion $(a)$ of Theorem~\ref{slanted}. Similarly, the assumption on convexity of the sets $E$ and $F$ in this theorem cannot be omitted. Indeed, let $E$ be the $y$-axis of $\R^2$ and $F \subseteq \R^2$ be the parabola given by the equation $y = x^2$. Then any vector $u \in \R^2$ is uniquely expressible as $x = y + z$, where $x \in E$ and $z \in F$, while $E$ and $F$ are not complementary planes of $\R^2$.

\end{enumerate}

\section{Proof of Theorem~\ref{decomposition}}

$(a) \Rightarrow (b)$. First, we consider the trivial case when at least one of the sets $E$ and $F$, say $E$, equals $\{ o \}$. Choose any vector $u \in H$. By the assumption,
\[
u = p_\fE(u) + p_\fF(u) = o + p_\fF(u) = p_\fF(u) \in F.
\]
Consequently, $F = H = \{ o \}^\circ = E^\circ$ and $E = \{ o \} = H^\circ = F^\circ$, as desired.

So, we may suppose that both sets $E$ and $F$ are distinct from $\{ o \}$. Choose a nonzero vector $u \in F$. Then $u = p_\fF(u)$, which gives $p_\fE(u) = u - p_\fF(u) = o$. Hence $o$ is the vector in $E$ nearest to $u$, and $E$ lies in the closed halfspace
\begin{equation} \label{eqn-decomposition}
V = \{ x \in H : \langle x, u \rangle \leqslant  0 \}.
\end{equation}
Thus,
\[
o \in E \subseteq \{ x \in H : \langle x, u \rangle \leqslant  0 \ \ \forall\, u \in F \} = F^\circ.
\]
Similarly, $o \in F \subseteq E^\circ$.

For the opposite inclusion, $F^\circ \subseteq E$, choose any vector $u \in F^\circ$. Then $F$ lies in the halfspace \eqref{eqn-decomposition}. Clearly, $p_\fV (u) = o$. Since $o \in F \subseteq V$, we also have $p_\fF (u) = o$. Therefore,
\[
u = p_\fE(u) + p_\fF(u) = p_\fE(u) + o = p_\fE(u) \in E.
\]
So, $F^\circ \subseteq E$. Summing up, $E = F^\circ$, which shows that $E$ is a closed convex cone. By a similar argument, $F$ is a closed convex cone, with $F = E^\circ$.

$(b) \Rightarrow (a)$. This part follows from Theorem~\ref{moreau}.

\section{Some Auxiliary Results}

This section describes the notation and auxiliary results on convex sets and cones in $\R^n$ which are necessary for the proofs of Theorems~\ref{projections}--\ref{slanted}. To avoid confliction of notation, the standard \textit{dot product} of vectors $x, y \in \R^n$ will be denoted by $x \! \cdot \! y$ (instead of $\langle x, y \rangle$). Vectors $x$ and $y$ are called \textit{orthogonal} (notation $x \bot y$) if $x \! \cdot \! y = 0$. Given distinct vectors $x, z \in \R^n$, by $[x, z]$ and $(x, z)$ we will mean the \textit{closed} and \textit{open} (line) \textit{segments} with end vectors $x, z \in \R^n$, defined, respectively, by
\[
[x, z] = \{ (1 - \lambda) x + \lambda z : 0 \leqslant \lambda \leqslant 1 \}, \quad (x, z) = \{ (1 - \lambda) x + \lambda z : 0 < \lambda < 1 \}.
\]

The \textit{orthogonal complement} of a nonempty set $X \subseteq \R^n$ is defined by
\[
X^\bot = \{ x \in \R^n : x \! \cdot \! u = 0 \ \ \forall\, u \in X \},
\]
while $\Span X$ denotes the (linear) \textit{span} of $X$. Obviously, $(\Span X)^\bot = X^\bot$. Furthermore, $\Dim X$, $\Cl X$, $\Int X$, and $\Bd X$, stand, respectively, for the dimension, closure, interior, and boundary of $X$. The set
\[
\Pos X = \{ \lambda_1 x_1 + \dots + \lambda_k x_k : k \geqslant 1, \, \lambda_i \geqslant 0, \, x_i \in X, \ 1 \leqslant  i \leqslant  k \}
\]
is called the \textit{positive hull} of $X$. It is well known (see, e.\,g., \cite{s-w70}, page 32) that $\Pos X$ is the smallest convex cone with apex $o$ which contains $X$.

If $L \subseteq \R^n$ is a plane of positive dimension, then a \textit{closed halfplane} $P$ of $L$ is the set of the form $P = L \cap V$, where $V$ is a closed halfspace of $\R^n$ satisfying the condition $\varnothing \ne L \cap V \ne L$. If a convex set $F \subseteq \R^n$ is not $n$-dimensional (that is, if $\Aff F$ is not $\R^n$), then $\Rbd F$ and $\Rint F$ will denote the \textit{relative boundary} and \textit{relative interior} of $F$.

It is assumed that the reader is familiar with basic concepts and results of convex analysis. Therefore, various standard results on convex sets are often given here without specific references (see, e.\,g., \cite{sol15} for details). For reader's convenience, necessary properties of convex cones are given below as the list of propositions (P1)--(P8). In these propositions, $C$ denotes a closed convex cone in $\R^n$.

\begin{enumerate}

\item[(P1)] The set $\Lin C = C \cap (-C)$, called the \textit{lineality space} of $C$, is the largest subspace contained in $C$.

\item[(P2)] $C \cap C^\circ = \{ o \}$, $(C^\circ)^\circ = C$, $\Span C = (\Lin C^\circ)^\bot$, and $\Span C^\circ = (\Lin C)^\bot$.

\item[(P3)] $C$ is a subspace if and only if any of the following conditions holds: (i) $C = \Lin C$, (ii) $o \in \Rint C$, (iii) $\Lin C \cap \Rint C \ne \varnothing$, (iv) $C^\circ$ is a subspace, (v) $C^\circ = C^\bot$.

\item[(P4)] If $C$ is not a subspace, then a vector $u \in \R^n$ belongs to $\Rint C$ if and only if $u \! \cdot \! v < 0$ for every vector $v \in C^\circ \setminus \Lin C^\circ$.

\item[(P5)] If $X \subseteq \R^n$ is a nonempty set, then $\Cl (\Pos X)$ is the smallest closed convex cone containing $X$. Furthermore,
\[
X^\circ = (\Pos X)^\circ = (\Cl (\Pos X))^\circ \quad \text{and} \quad \Cl (\Pos X) = (X^\circ)^\circ
\]

\item[(P6)] For a nonempty convex set $E \subseteq \R^n$, the following conditions are equivalent: (i) $o \in \Rint E$, (ii) $\Pos E$ is a subspace, (iii) $\Cl (\Pos E)$ is a subspace.

\item[(P7)] If a hyperplane $H \subseteq \R^n$ supports $C$, then $\Lin C \subseteq H$.

\item[(P8)] (see~\cite{sol18}) If $C_1$ and $C_2$ are closed convex cones in $\R^n$ satisfying the condition
\[
C_1 \cap C_2 = \Lin C_1 \cap \Lin C_2,
\]
then there is an $(n - 1)$-dimensional subspace $S \subseteq \R^n$ separating $C_1$ and $C_2$ such that
\[
C_1 \cap S = \Lin C_1 \quad \text{and} \quad C_2 \cap S = \Lin C_2.
\]

\end{enumerate}

The proof of Theorem~\ref{projections} uses the following result of proper interest.

\begin{theorem} \label{aux}
Let $C \subseteq \R^n$ be a closed convex cone distinct from a subspace, $B$ be a convex cone which lies in $\Rbd C$, and $D$ be the orthogonal complement of $B$ in $C^\circ$:
\[
D = \{ x \in C^\circ : x \! \cdot \! u = 0 \ \ \forall\, u \in B \}.
\]
The following assertions hold.
\begin{enumerate}

\item[$(a)$] $D$ is a closed convex cone.

\item[$(b)$] If $B \subseteq \Lin C$, then $D = C^\circ$, and if $D \not\subseteq \Lin C$, then $D$ lies in $\Rbd C^\circ$ but not in $\Lin C^\circ$.

\item[$(c)$] There is an $(n - 1)$-dimensional subspace $S \subseteq \R^n$ containing $B$ and separating $C$ and $D$ such that $\Rint C$ and $\Rint D$ lie in the opposite open halfspaces of\, $\R^n$ determined by $S$.

\end{enumerate}
\end{theorem}

\begin{proof}
Because $C$ is not a subspace, (P3) shows that $C^\circ$ is not a subspace.

$(a)$ $D$ is a closed convex cone as the intersection of $C^\circ$ and the subspace $B^\bot$.

$(b)$ If $B \subseteq \Lin C$, then, by (P2), $C^\circ \subseteq (\Lin C)^\bot \subseteq B^\bot$. Therefore,
\[
C^\circ \subseteq B^\bot \cap C^\circ = D \subseteq C^\circ,
\]
which gives $D = C^\circ$.

Let $B \not\subseteq \Lin C$. Since $B$ is a convex subset of $\Rbd C$, there is a hyperplane $H \subseteq \R^n$ which contains $B$ and supports $C$ such that $C \not\subseteq H$ (see, e.\,g., \cite{sol15}, Corollary~6.9). By (P7), $H$ contains $\Lin C$ and whence $H$ is an $(n - 1)$-dimensional subspace. Denote by $V$ the closed halfspace of $\R^n$ determined by $H$ and containing $C$ and express it as
\[
V = \{ x \in \R^n : x \! \cdot \! e \leqslant 0 \},
\]
where $e$ is a suitable nonzero vector. Choose a vector $y \in \Rint B \setminus \Lin C$ (clearly, $\Rint B \not\subseteq \Lin C$, since otherwise $B \subseteq \Cl (\Rint B) \subseteq \Lin C$, contrary to the assumption $B \not\subseteq \Lin C$). Let $u = e + y$. Obviously, $p_\fV (u) = y$. Therefore, $p_\fE (u) = y$ due to $y \in E \subseteq V$.

We contend that $e \in D$. Indeed, denote by $z$ the metric projection of $u$ on $C^\circ$. By Theorem~\ref{moreau}, $y \bot z$ and $u = y + z$. So, $e = u - y = z \in C^\circ$. Because $e \in H^\bot \subseteq B^\bot$, we obtain the inclusion $e \in C^\circ \cap B^\bot = D$.

If $D$ contained a vector $v \in \Rint C^\circ$, then (since $C^\circ$ is not a subspace) we would have $v \! \cdot \! y< 0$ due to (P4), formulated for $C^\circ$, and the choice of $y \in C \setminus \Lin C$. Hence $D \subseteq \Rbd C^\circ$. Also, $D$ does not lie in $\Lin C^\circ$. Indeed, assuming that $D \subseteq \Lin C^\circ$, we obtain that $z \in \Lin C^\circ$, and (P2) gives
\[
C \subseteq \Span C = (\Lin C^\circ)^\bot \subseteq \{ z \}^\bot = H,
\]
contrary to $C \not\subseteq H$. By (P1), $\Lin C^\circ$ is the largest subspace contained in $C^\circ$. Therefore, $D$ is not a subspace (otherwise $D \subseteq \Lin C^\circ$).

$(c)$ If $B \subseteq \Lin C$, then $D = C^\circ$ by assertion $(b)$, and the existence of a desired subspace $S$ follows from (P8) and the equality $C \cap C^\circ = \{ o \}$. Suppose that $B \not\subseteq \Lin C$. In what follows, we use the notation from part $(b)$.

Let $L = \Span D$ and $M = \Span B$. Because $B$ and $D$ are orthogonal sets, it is obvious that their spans are orthogonal subspaces. Furthermore, $M \subseteq H$ due to $B \subseteq H$. Denote by $C'$ the orthogonal projection of $C$ on $L$. Clearly, $C'$ is a convex (not necessarily closed) cone which lies in the closed halfplane $L \cap V$ of $L$. Since $C \not\subseteq H$, the cone $C'$ is properly supported by the subspace $L \cap H$, whose dimension is $\Dim L - 1$, implying, by (P3), that $o \notin \Rint C'$ and $C'$ is not a subspace.

Any vector $x \in C'$ is the orthogonal projection of a suitable vector $p \in C$, which is expressible as $p = x + q$, with $q \in L^\bot$. Given a vector $v \in D$, one has
\[
x \! \cdot \! v = (p - q) \! \cdot \! v = p \! \cdot \! v + 0 \leqslant  0
\]
due to $D \subseteq C^\circ \cap L$. If $D^\circ_\fL$ denotes the polar cone of $D$ in the space $L$, then the above argument shows that $C'$ lies in $D^\circ_\fL$. Consequently, $\Cl C' \subseteq D^\circ_\fL$, which gives
\[
\{ o \} \subseteq \Lin D \cap \Lin (\Cl C') \subseteq D \cap \Cl C' \subseteq D \cap D^\circ_\fL = \{ o \}.
\]
So,
\[
\Lin D \cap \Lin (\Cl C') = D \cap \Cl C' = \{ o \}.
\]

By (P8), formulated for the space $L$, there is a subspace $N$ of $L$ of dimension $\Dim L - 1$ separating $D$ and $\Cl C'$ such that
\[
N \cap D = \Lin D \quad \text{and} \quad N \cap \Cl C' = \Lin (\Cl C').
\]
(Possibly, $N \ne H \cap L$.) Since neither $D$ nor $\Cl C'$ is a subspace, one has $D \ne \Lin D$ and $\Cl C' \ne \Lin (\Cl C')$, as follows from (P3). Consequently, $D \not\subseteq N$ and $\Cl C' \not\subseteq N$.

Denote by $Q$ and $Q'$ the closed halfplanes of $L$ which are determined by $N$ and contain the cones $D$ and $\Cl C'$, respectively. Let
\[
S = N + L^\bot, \quad P = Q + L^\bot, \quad P' = Q' + L^\bot.
\]
Clearly, $S$ is an $(n - 1)$-dimensional subspace containing $B$ (due to $B \subseteq L^\bot$), and $P$ and $P'$ are closed halfspaces determined by $S$ and containing $D$ and $\Cl C' + L^\bot$, respectively. By the above argument,
\[
C \subseteq C' + L^\bot \subseteq Q' + L^\bot = P'.
\]
Furthermore, neither $D$ nor $C$ lies in $S$. The latter argument implies the inclusions $\Rint D \subseteq \Int P$ and $\Rint C \subseteq \Int P'$, as desired.
\end{proof}

The following example illustrates Theorem~\ref{aux}. 

\begin{example} 
Let $C$ be a planar closed convex cone in $\R^3$, given by $C = \{ (x, y, 0) : x \geqslant 3 |y| \}$, and let $B = \{ (y, 3y, 0) : y \geqslant 0 \}$ be the boundary halfline of $C$. Then $C^\circ = \{ (x, y, z) : x \leqslant -|y|/3 \}$ and $D = \{ (-y/3, y, z) : y \geqslant 0 \}$ is the vertical closed halfplane bounded by the $z$-axis. 

Clearly, every 2-dimensional subspace through $B$ supports $C$, but only one of them, $S = \{ (y, 3y, z) : y \in \R \}$, separates $C$ and $D$. Furthermore, $\Rint C$ and $\Rint D$ lie in the opposite open halfspaces of\, $\R^3$ determined by $S$. One can easily see, that $S$ does not separate $C$ and $C^\circ$. 
\end{example} 

\section{Proof of Theorem~\ref{projections}}

Since the implication $(b) \Rightarrow (a)$ immediately follows from Theorem~\ref{moreau}, it suffices to show that $(a) \Rightarrow (b)$. First, we exclude the obvious case when $E = \{ o \}$ or $F = \{ o \}$. Indeed, suppose that $E = \{ o \}$. Then the condition $E + F = \R^n$ gives $F = \R^n$. Consequently, $E = F^\circ$ and $F = E^\circ$. The case $F = \{ o \}$ is similar.

So, we may assume that neither $E$ nor $F$ is $\{ o \}$. Our further argument is divided into a sequence of lemmas, where the sets $E$ and $F$ are assumed to satisfy assertion $(a)$ of the theorem.

\begin{lemma} \label{projections-1}
$o \in E \cap F$.
\end{lemma}

\begin{proof}
Assume, for instance, that $o \notin E$. Let $e = p_\fE (o)$. Then $e \ne o$ and $E$ lies in the closed halfspace
\[
W = \{ x \in \R^n : (x - e) \! \cdot \! e \geqslant 0 \}.
\]
Put $c = p_\fF (o)$. By the hypothesis, $e \bot c$. Hence $c$ belongs to the $(n - 1)$-dimensional subspace
\[
S = \{ x \in \R^n : x \! \cdot \! e = 0 \}.
\]

Choose a scalar $t < 0$. Since $p_\fW (te) = e$, the inclusions $e \in E \subseteq W$ give $p_\fE (te) = e$. Then $p_\fF (te)$ belongs to $S$ due to the assumption $p_\fE (te)\, \bot\, p_\fF (te)$. Because $o$ is the orthogonal projection of $te$ on $S$, the Pythagorean theorem, used twice, gives
\begin{align*}
\| t e - p_\fF (te) \|^2 & = \| t e - o \|^2 + \| o - p_\fF (te) \|^2 \geqslant \| t e - o \|^2 + \| o - p_\fF (o) \|^2 \\
& = \| t e - p_\fF (o) \|^2 = \| t e - c \|^2.
\end{align*}
Hence
\[
\| t e - p_\fF (te) \| = \| te - c \|
\]
by the definition of $p_\fF (te)$. The uniqueness of metric projection of $te$ on $F$ implies that $p_\fF(te) = c$. Obviously, $te \ne c$. Therefore, $F$ lies in the closed halfspace
\begin{align*}
V(t) & = \{ x \in \R^n : (x - c) \! \cdot \! (te - c) \leqslant  0 \} \\
& = \{ x \in \R^n : (x - c) \! \cdot \! (e - c/t) \geqslant 0 \}
\end{align*}
as $t < 0$. Letting $t \to -\infty$, we conclude that $F$ lies in the closed halfspace
\[
V = \{ x \in \R^n : (x - c) \! \cdot \! e \geqslant 0 \} = \{ x \in \R^n : x \! \cdot \! e \geqslant 0 \}.
\]
Consequently,
\[
E + F \subseteq W + V = V,
\]
contrary to the assumption $E + F = \R^n$. Thus $o \in E \cap F$.
\end{proof}

\begin{lemma} \label{projections-2}
$E \subseteq F^\circ$ and\, $F \subseteq E^\circ$.
\end{lemma}

\begin{proof}
Choose a nonzero vector $u \in F$ (we already excluded the case $F = \{ o \}$). Since $u = p_\fF(u)$, the orthogonality of $p_\fE(u)$ and $p_\fF(u)$ shows that $p_\fE(u)$ belongs to the $(n - 1)$-dimensional subspace
\[
S = \{ x \in \R^n : x \! \cdot \! u = 0 \}.
\]

Obviously, $o$ is the orthogonal projection of $u$ on $S$. So,
\[
\| u - o \| \leqslant  \| u - p_\fE(u) \|.
\]
The inclusion $o \in E$ (see Lemma~\ref{projections-1}) and the uniqueness of $p_\fE(u)$ imply that $p_\fE(u)= o$. Therefore, $E$ lies in the closed halfspace
\[
V = \{ x \in \R^n : x \! \cdot \! u \leqslant  0 \}.
\]
Consequently,
\[
E \subseteq \{ x \in \R^n : x \! \cdot \! u \leqslant  0 \ \ \forall\, u \in F \} = F^\circ.
\]
The proof of the inclusion $F \subseteq E^\circ$ is similar.
\end{proof}

\begin{lemma} \label{projections-3}
If $o \in \Rint E \cup \Rint F$, then $E$ and $F$ are complementary orthogonal subspaces: $E = F^\bot$ and $F = E^\bot$.
\end{lemma}

\begin{proof}
Let, for instance, $o \in \Rint E$. By (P6), the cone $C = \Pos E$ is a subspace. A combination of (P3) and (P5) gives $E^\circ = C^\circ = C^\bot$, and Lemma~\ref{projections-2} shows that $F$ lies in the subspace $C^\bot$.

We assert that $E = C$. Indeed, assume for a moment that $E \ne C$. Then there is a closed halfplane $P$ of $C$ which contains $E$. Let
\[
P = \{ x \in C : x \! \cdot \! e \leqslant  \gamma \},
\]
where $e$ is a suitable nonzero vector in $C$ and $\gamma$ is a scalar. Then the sum $P + C^\bot$ coincides with the closed halfspace
\[
V = \{ x \in \R^n : x \! \cdot \! e \leqslant  \gamma \}.
\]
Consequently,
\[
E + F \subseteq P + C^\bot = V,
\]
contrary to the hypothesis $E + F = \R^n$. So, $E = C$, implying that $E$ is a subspace.

In a similar way, assuming that $F \ne C^\bot$, we would find a closed halfplane $Q$ of $C^\bot$ which contains $F$, resulting in the inclusion
\[
E + F \subseteq C + Q \ne C + C^\bot = \R^n.
\]
Hence $F$ should coincide with $C^\bot$, as desired.
\end{proof}

A combination of (P3) and Lemma~\ref{projections-3} implies the following corollary.

\begin{corollary} \label{cor-1}
If at least one of the sets $E$ and $F$ is a subspace, then $E$ and $F$ are complementary orthogonal subspaces.
\end{corollary}

\begin{lemma} \label{projections-4}
$E \not\subseteq \Rbd F^\circ$ and $F \not\subseteq \Rbd E^\circ$.
\end{lemma}

\begin{proof}
Assume, for contradiction, that $E \subseteq \Rbd F^\circ$. Then $F^\circ$ is not a subspace (otherwise $\Rbd F^\circ = \varnothing$), and a combination of (P3) and Corollary~\ref{cor-1} shows that none of the sets $E, F, E^\circ$, and $F^\circ$ is a subspace.

First, we contend that $\Rint E \not\subseteq \Lin F^\circ$. Indeed, let $\Rint E \subseteq \Lin F^\circ$. Then $E = \Cl (\Rint E) \subseteq \Lin F^\circ$. Obviously, both subspaces $\Lin F^\circ$ and $\Span F = (\Lin F^\circ)^\bot$ are nontrivial. Since $F \ne \Span F$ (otherwise $F$ is a subspace), we have
\[
E + F \subseteq \Lin F^\circ + F \ne \Lin F^\circ + \Span F = \Lin F^\circ + (\Lin F^\circ)^\bot = \R^n,
\]
contrary to the assumption $E + F = \R^n$. So, $\Rint E \not\subseteq \Lin F^\circ$.

Choose a nonzero vector $y \in \Rint E \setminus \Lin F^\circ$ and consider the closed halfline $B = \{ \lambda y : \lambda \geqslant 0 \}$. Clearly, $B \subseteq \Rbd F^\circ$ and $B \not\subseteq \Lin F^\circ$ due to the choice of $y$ in $\Rbd F^\circ \setminus \Lin F^\circ$.

By Theorem~\ref{aux} (applied to the case $C = F^\circ$), there is an $(n - 1)$-dimen\-sional subspace $S \subseteq \R^n$ which contains $B$ and supports $F^\circ$ such that the relative interiors of the closed convex cones $F^\circ$ and
\[
D = \{ x \in (F^\circ)^\circ : x \! \cdot \! b = 0 \ \ \forall\, b \in B \} = \{ x \in (F^\circ)^\circ : x \! \cdot \! y = 0 \}
\]
lie in the opposite open halfspaces of\, $\R^n$ determined by $S$. Since $y \in \Rint E$ and $S$ supports $E$, we have $E \subseteq S$. Express $S$ as
\[
S = \{ x \in \R^n : x \! \cdot \! e = 0 \},
\]
where $e \ne o$ is a suitable vector. Without loss of generality, we may assume that $\Rint F^\circ$ lies in the open halfspace
\[
W = \{ x \in \R^n : x \! \cdot \! e > 0 \}
\]
(otherwise replace $e$ with $-e$). Clearly, $D$ lies in the closed halfspace $V = \R^n \setminus W$. If $V$ contained $F$, then
\[
E + F \subseteq S + V = V \ne \R^n,
\]
which is impossible by the assumption $E + F = \R^n$. Hence there is a nonzero vector $v \in F \cap W$. Clearly, $v \! \cdot \! e > 0$ due to $v \in W$, and $v \! \cdot \! y \leqslant 0$ due to the inclusions $v \in F$ and $y \in E \subseteq F^\circ$ .

Denote by $h$ the closed halfline $\{ y + t e : t \geqslant 0 \}$. The hyperplane
\[
H = \{ x \in \R^n : (x - v) \! \cdot \! v = 0 \}
\]
meets $h$ at the vector $y + t_0 e$, where
\[
t_0 = \frac{v \! \cdot \! v - y \! \cdot \! v}{e \! \cdot \! v} > 0
\]
(indeed, $(y + t_0 e - v) \! \cdot \! v = 0$). Put $u = y + t_0 e$. Since the halfline $h$ is normal to $S$, and since $y \in E \subseteq S$, we have $p_\fE (u) = y$. The orthogonality of $u - v$ and $v$ and the Pythagorean theorem give
\[
\| u - v \|^2 + \| v \|^2 = \| u \|^2,
\]
which results in the inequality $\| u - v \| < \| u \|$. Since $v \in F$ and $p_\fF (u)$ is the vector in $F$ nearest to $u$, we have
\begin{equation} \label{1}
\| u - p_\fF (u) \| \leqslant \| u - v \| < \| u \|.
\end{equation}

Next, we are going to prove that $\| u - x \| \geqslant \| u \|$ whenever $x \in D$. Indeed, let $x \in D$. Then $x \in V$ (and $y \in S = \{ e \}^\bot$), which gives
\[
(x - y) \! \cdot \! (u - y) = t_0 (x - y) \! \cdot \! e = t_0\, x \! \cdot \! e + 0 \leqslant 0.
\]
Hence
\begin{align*}
\| x - u \|^2 & = \| x - y \|^2 + \| y - u \|^2 + 2 (x - y) \! \cdot \! (u - y) \\
& \geqslant \| x - y \|^2 + \| y - u \|^2 \\
& = \| x \|^2 + \| y \|^2 + \| y - u \|^2 & (\text{as} \ x \! \cdot \! y = 0) \\
& = \| x \|^2 + \| u \|^2 & (\text{as} \ (u - y) \! \cdot \! y = 0) \\
& \geqslant \| u \|^2.
\end{align*}

Combining this argument with \eqref{1}, we conclude that $p_\fF (u) \notin D$. On the other hand, the vector $p_\fF (u)$ must be orthogonal to $y = p_\fE (u)$ by assumption $(a)$.
Consequently, $p_\fF (u)$ should satisfy the inclusion
\[
p_\fF (u) \in F \cap \{ y \}^\bot \subseteq (F^\circ)^\circ \cap \{ y \}^\bot = D.
\]
The obtained contradiction implies that the inclusion $E \subseteq \Rbd F^\circ$ does not hold. Similarly, $F \not\subseteq \Rbd E^\circ$.
\end{proof}

\begin{lemma} \label{projections-5}
$E$ and $F$ are polar cones of each other.
\end{lemma}

\begin{proof}
If $o \in \Rint E \cup \Rint F$, then Lemma~\ref{projections-3} shows that $E$ and $F$ are complementary orthogonal subspaces, polar cones of each other. So, we may assume that $o \notin \Rint E \cup \Rint F$. By (P6), none of the cones $\Pos E$ and $\Pos F$ is a subspace, and a combination (P3) and (P5) shows that none of $E^\circ$ and $F^\circ$ is a subspace.

We will prove the equality $E = F^\circ$ (the case $F = E^\circ$ is similar). Assume, for contradiction, that $E \ne F^\circ$. By Lemmas~\ref{projections-2} and~\ref{projections-4}, $E \subseteq F^\circ$ and $E \not\subseteq \Rbd F^\circ$. Therefore, $E$ meets $\Rint F^\circ$. Let $L = \Span F^\circ$. We are going to prove the following auxiliary assertion:

\begin{enumerate}
\item[$(c)$] there is a closed halfplane $P$ of $L$ which contains $E$ such that the boundary plane of $P$ supports $E$ at a nonzero vector from $\Rint F^\circ$.
\end{enumerate}

To prove $(c)$, we consider separately the cases $\Dim E = \Dim F^\circ$ and $\Dim E < \Dim F^\circ$.

1. Let $\Dim E = \Dim F^\circ$. Then $\Span E = L$ and $\Rint E \subseteq \Rint F^\circ$ (see \cite{sol15}, Theorem~2.28). Furthermore, $\Rint F^\circ \not\subseteq E$, since otherwise $F^\circ = \Cl (\Rint F^\circ) \subseteq E$. Choose vectors $x \in \Rint F^\circ \setminus E$ and $z \in \Rint E$. Then the open segment $(x, z)$ meets $\Rbd E$ at a vector $y$, say (see \cite{sol15}, Theorem~2.55). The inclusions $x \in \Rint F^\circ$ and $z \in E \subseteq F^\circ$ imply that $y \in (x, z) \subseteq \Rint F^\circ$. So, $y \in \Rbd E \cap \Rint F^\circ$. One has $y \ne o$ because $o \notin \Rint F^\circ$ due to (P3). Choose a closed halfplane $P$ of $L$ which contains $E$ such that the boundary plane of $P$ supports $E$ at $y$.

2. Let $\Dim E < \Dim F^\circ$. Since both sets $E$ and $F^\circ$ contain $o$, the subspace $M = \Span E$ is a proper subset of $L$. By (P2), the subspaces $\Lin (F^\circ)^\circ$ and $L$ are orthogonal complements of each other. Consequently, $\Lin (F^\circ)^\circ + M$ is a proper subspace of $\R^n$. Choose an $(n - 1)$-dimensional subspace $H$ which contains $\Lin (F^\circ)^\circ + M$ and let $N = H \cap L$. Then
\[
\Dim N = \Dim H + \Dim L - \Dim (H + L) = (n - 1) + \Dim L - n = \Dim L - 1.
\]
Now, assertion $(c)$ holds for any choice of a nonzero vector $y \in E \cap \Rint F^\circ$ and of a closed halfplane $P$ of $L$ determined by $N$.

\medskip

We can describe the halfplane $P$ from $(c)$ as
\[
P = \{ x \in L : x \! \cdot \! e \leqslant \gamma \},
\]
where $e$ is a nonzero vector in $L$ and $\gamma = y \! \cdot \! e$. Obviously, $\gamma \geqslant 0$ due to $o \in E \subseteq P$ (see Lemma~\ref{projections-1}). Consider the closed halfspace
\[
V = P + L^\bot = \{ x \in \R^n : x \! \cdot \! e \leqslant \gamma \}.
\]
Then $E \subseteq P \subseteq V$ and the boundary hyperplane of $V$ supports $E$ at $y$.

Next, we contend that $F$  does not lie in the closed halfspace
\[
W = \{ x \in \R^n : x \! \cdot \! e \leqslant 0 \}.
\]
Indeed, if $F \subseteq W$, then $E + F \subseteq V + W = V$, contrary to the assumption $E + F = \R^n$.

This argument implies that $e \notin F^\circ$; indeed, otherwise
\[
F \subseteq (F^\circ)^\circ \subseteq \{ e \}^\circ = W.
\]
Consequently, the closed halfline $h = \{ y + t e : t \geqslant 0 \}$ lies in $L$ but not in $F^\circ$. By a convexity argument, there is a scalar $t_0 > 0$ such that the halfline $h_0 = \{ y + t e : t \geqslant t_0 \}$ is disjoint from $F^\circ$.

Since the halfline $h$ is normal to the boundary hyperplane of $V$ and $h \cap V = \{ y \}$, for every vector $u \in h_0$, we have $p_\fV (u) = y$, and hence $p_\fE (u) = y$ as $y \in E \subseteq V$. Also, the inclusion $E \subseteq F^\circ$ implies that $p_\fE (u) \in F^\circ$.

We contend the existence of a vector $u' \in h_0$ with the property $p_\fF (u') \notin \Lin (F^\circ)^\circ$. Indeed, assume for a moment that $p_\fF (u) \in \Lin (F^\circ)^\circ$ for all $u \in h_0$. By (P2), the subspaces $\Lin (F^\circ)^\circ$ and $L$ are orthogonal. Therefore, $u\, \bot\, p_\fF (u)$ whenever $u \in h_0$. Pythagorean's theorem gives
\[
\| u - p_\fF (u) \|^2 = \| u - o \|^2 + \| o -  p_\fF (u) \|^2 \geqslant \| u - o \|^2.
\]
Because $o \in F$ (see Lemma~\ref{projections-1}), the uniqueness of $p_\fF (u)$ gives $p_\fF (u) = o$. Consequently, writing $u$ as $y + t e$, $t \geqslant t_0$, we see that $F$ lies in every closed halfspace
\[
W(t) = \{ x \in \R^n : x \! \cdot \! (y + te) \leqslant 0 \} = \{ x \in \R^n : x \! \cdot \! (e + y/t) \leqslant 0 \}, \quad t \geqslant t_0 > 0.
\]
Letting $t \to \infty$, we obtain that $F$ lies in $W$, which is impossible by the above argument.
Hence, there is a vector $u' \in h_0$ satisfying the condition $p_\fF (u') \in F \setminus \Lin (F^\circ)^\circ$.

Finally, (P4) implies that the vectors
\[
p_\fE (u') = y \in \Rint F^\circ \quad \text{and} \quad p_\fF (u') \in F \setminus \Lin (F^\circ)^\circ \subseteq (F^\circ)^\circ \setminus \Lin (F^\circ)^\circ
\]
cannot be orthogonal, contrary to assertion $(a)$ of  Theorem~\ref{projections}. The obtained contradiction shows that $E = F^\circ$.
\end{proof}

\section{Proof of Theorem~\ref{orthogonal}}

Since the implication $(b) \Rightarrow (a)$ immediately follows from Theorem~\ref{moreau}, it suffices to show that $(a) \Rightarrow (b)$. Clearly, $E + F = \R^n$ due to condition $(a)$.

We may exclude the obvious cases when any of $E$ and $F$ is $\{ o \}$ or $\R^n$. Indeed, suppose that $E = \{ o \}$. Choose a vector $u \in \R^n$. By the assumption, $u$ is uniquely expressible as $u = y + z$, where $y \in E$ and $z \in F$. Since $y = o$, we have $u = z \in F$. So, $F = \R^n = \{ o \}^\circ = E^\circ$, as desired. The case $E = \R^n$ is similar.

We also observe that $E \cap F = \{ o \}$. Indeed, by condition $(a)$, we can write $o = y + z$ for suitable orthogonal vectors $y \in E$ and $z \in F$. Then $0 = \| o \|^2 = \| y \|^2 + \| z \|^2$, which gives $o = y = z \in E \cap F$. If there were a nonzero vector $u \in E \cap F$, then $u = u + o$ and $u = o + u$ would be two different representations of $u$ as a sum of orthogonal vectors from $E$ and $F$, respectively. So, $E \cap F = \{ o \}$.

We contend that both sets $E$ and $F$ are cones. Indeed, choose any vector $u \in E$ and a scalar $\lambda > 0$. If $0 \leqslant \lambda \leqslant 1$, then $\lambda u \in [o, u] \subseteq E$ due to the convexity of $E$. Let $\lambda > 1$. By condition $(a)$, there are orthogonal vectors $y \in E$ and $z \in F$ such that $\lambda u = y + z$. Thus $u = \lambda^{-1} y + \lambda^{-1} z$ is the sum of orthogonal vectors $\lambda^{-1} y$ and $\lambda^{-1} z$. Because $0 < \lambda^{-1} < 1$, one has $\lambda^{-1} y \in [o, y] \subseteq E$ and $\lambda^{-1} z \in [o, z] \subseteq F$ due to the convexity of $E$ and $F$. On the other hand, $u = u + o$ is the sum of orthogonal vectors $u \in E$ and $o \in F$. The uniqueness of such expression gives $u = \lambda^{-1} y$. Consequently, $\lambda u = y \in E$. Hence $E$ is a cone.

Next, we assert that $E^\circ \subseteq F$ and $F^\circ \subseteq E$. Indeed, choose any nonzero vector $u \in E^\circ$ ($E^\circ \ne \{ o \}$ because $E \ne \R^n$). By condition $(a)$, $u = y + z$, where $y \in E$ and $z \in F$ are orthogonal vectors. Then
\[
u \! \cdot \! y = (y + z) \! \cdot \! y = y \! \cdot \! y + z \! \cdot \! y = \| y \|^2 + 0 \geqslant 0,
\]
with $u \! \cdot \! y = 0$ if and only if $y = o$. On the other hand $u \! \cdot \! y \leqslant 0$ due to $y \in E$ and $u \in E^\circ$. Hence $u \! \cdot \! y$, implying $y = o$. Thus  $u = o + z = z \in F$. So, $E^\circ \subseteq F$. In a similar way, $F^\circ \subseteq E$.

Finally, we contend that $F \subseteq E^\circ$ and $E \subseteq F^\circ$. For this, choose any vector $u \in F$. By Theorem~\ref{moreau}, $u = y + z$, where $y \in E$ and $z \in E^\circ$ are orthogonal vectors. The above proved inclusion $E^\circ \subseteq F$ gives $z \in F$. On the other hand $u = o + u$ is the sum of orthogonal vectors $o \in E$ and $u \in F$. Condition $(a)$ implies that these representation are identical: $y = o$ and $u = z \in E^\circ$. So, $F \subseteq E^\circ$. Similarly, $E \subseteq F^\circ$.

Summing up, $E = F^\circ$ and $F = F^\circ$, as desired.

\section{Proof of Theorem~\ref{slanted}}

$(a) \Rightarrow (b)$. First, we exclude the trivial case when $E = \{ o \}$ or $F = \{ o \}$. Indeed, suppose that $E = \{ o \}$. Choose a vector $u \in \R^n$. By the assumption, $u$ is uniquely expressible as $u = y + z$, where $y \in E$ and $z \in F$. Since $y = o$, we have $u = z \in F$. So, $F = \R^n$, as desired. Since the case $E = \{ o \}$ is similar, we may assume that neither $E$ nor $F$ is $\{ o \}$.

By the assumption, the vector $o$ is uniquely expressible as $o = c - c$, where $c \in E$ and $-c \in F$. Let $E_0 = E - c$ and $F_0 = F + c$. Then $o \in E_0 \cap F_0$.

We assert that the sets $E_0$ and $F_0$ satisfy condition $(a)$ of the theorem. Indeed, choose any vector $u \in \R^n$. Condition $(a)$ shows that $u = y' + z'$ for suitable vectors $y' \in E$ and $z' \in F$. Consequently, $u = (y' - c) + (z' + c)$, where $y' - c \in E_0$ and $z' + c \in F_0$. For the uniqueness of such a representation, let $u = y_1 + z_1$ and $u = y_2 + z_2$ for suitable vectors $y_1, y_2 \in E_0$ and $z_1, z_2 \in F_0$. Put $y_i' = y_i + c$ and $z_i' = z_i - c$, $i = 1, 2$. Clearly, $y_i' \in c + E_0 = E$ and $z_i' \in F_0 - c = F$, $i = 1, 2$. Since $u = y_1' + z_1' = y_2' + z_2'$, condition $(a)$ gives $y_1' = y_2'$ and $z_1' = z_2'$. Hence
\[
y_1 = y_1' - c = y_2' - c = y_2, \quad z_1 = z_1' + c = z_2' + c = z_2.
\]
Summing up, $E_0$ and $F_0$ satisfy condition $(a)$ of the theorem. In particular, $E_0 + F_0 = \R^n$.

Next, we observe that $E_0 \cap F_0 = \{ o \}$. Indeed, assume for a moment the existence of a nonzero vector $u \in E_0 \cap F_0$. Then $u = u + o = o + u$ are two distinct expressions of $u$ as sums of vectors from $E_0$ and $F_0$, respectively, in contradiction with condition $(a)$. Hence $E_0 \cap F_0 = \{ o \}$.

By a slight modification of the argument given in the proof of Theorem~\ref{orthogonal}, one can show that $E_0$ and $F_0$ are both cones. We divide our further argument into two cases: $o \in \Rint E_0 \cup \Rint F_0$ and $o \notin \Rint E_0 \cup \Rint F_0$. We are going to prove that the first case holds while the second one is impossible.

1. Assume first that $o \in \Rint E_0 \cup \Rint F_0$. We are going to show that under this assumption, $E_0$ and $F_0$ are complementary subspaces. Indeed, let, for instance, $o \in \Rint E_0$. Then the cone $E_0$ is a subspace, as follows from (P3).

We further contend that $o \in \Rint F_0$. Indeed, suppose that $o \notin \Rint F_0$. In this case, the equality $E_0 \cap F_0 = \{ o \}$ implies that $\Rint E_0 \cap \Rint F_0 = \varnothing$. Consequently, there is an $(n - 1)$-dimensional subspace $H \subseteq \R^n$ which contains the subspace $E_0$ and supports the convex set $F_0$ (see, e.\,g., \cite[Theorem~6.8]{sol15}). If $V \subseteq \R^n$ is a closed halfspace containing $F_0$ and determined by $H$, then $E_0 + F_0 \subseteq H + V = V$, contrary to $E_0 + F_0 = \R^n$. Hence $o \in \Rint F_0$.

As above, the inclusion $o \in \Rint F_0$ shows that the cone $F_0$ is a subspace. Consequently, $E_0$ and $F_0$ are complementary subspaces due to the properties $E_0 + F_0 = \R^n$ and $E_0 \cap F_0 = \{ o \}$.

2. Assume, for contradiction, that $o \notin \Rint E_0 \cup \Rint F_0$. Then $o \in \Rbd E_0 \cap \Rbd F_0$. Since $o \in -F_0$ and $\Rint (-F_0) = -\Rint F_0$, one has $o \in \Rbd (-F_0)$.

Under this condition, we contend that $E_0 \cap (-F_0) = \{ o \}$. Indeed, if $E_0 \cap (-F_0)$ contained a nonzero vector $u$, then $o = o + o$ and $o = u + (-u)$ would be distinct representations of $o$ as sums of vectors from $E_0$ and $F_0$, respectively.

Hence $E_0 \cap (-F_0) = \{ o \}$. Consequently, there is an $(n - 1)$-dimensional subspace $H \subseteq \R^n$ which separates $E_0$ and $-F_0$ (see, e.\,g., \cite[Theorem~6.30]{sol15}). If $V \subseteq \R^n$ is a closed halfspace determined by $H$ and containing $E_0$, then $-F_0 \subseteq -V$, implying the inclusion $F_0 \subseteq V$. Therefore, $E_0 + F_0 \subseteq V + V = V$, contrary to $E_0 + F_0 = \R^n$.

The obtained contradiction implies that the case $o \notin \Rint E_0 \cup \Rint F_0$ does not hold. Summing up, the sets $E_0$ and $F_0$ are complementary subspaces. Therefore, the sets $E = c + E_0$ and $F = F_0 - c$ are complementary planes.

$(b) \Rightarrow (a)$. Suppose that $E$ and $F$ are complementary planes. Then $E = c_1 + S_1$ and $F = c_2 + S_2$, where $c_1, c_2 \in \R^n$ and $S_1$ and $S_2$ are complementary subspaces. Given a vector $u \in \R^n$, there exist unique vectors $x_1 \in S_1$ and $x_2 \in S_2$ such that $u - c_1 - c_2 = x_1 + x_2$. Hence $u = y + z$, where $y = c_1 + x_1 \in E$ and $z = c_2 + x_2 \in F$ is the unique representation of $u$.

\bigskip

\noindent \textbf{Acknowledgment.} The author thanks the anonymous referee for many helpful suggestions essentially improving the original version of the manuscript.

\end{document}